\UseRawInputEncoding
\documentclass[11pt]{article}
\usepackage{amsmath, amssymb, amsthm}
\usepackage{enumitem}
\usepackage[margin=1in]{geometry}

\newtheorem{theorem}{Theorem}[section]
\newtheorem{lemma}[theorem]{Lemma}
\newtheorem{proposition}[theorem]{Proposition}
\newtheorem{definition}[theorem]{Definition}

\title{Three-Edges and the SOS Rank of Biquadratic Forms}
\date{\today}

\begin{document}

\Large
       \author{Liqun Qi\footnote{Jiangsu Provincial Scientific Research Center of Applied Mathematics, Nanjing 211189, China.
			Department of Applied Mathematics, The Hong Kong Polytechnic University, Hung Hom, Kowloon, Hong Kong.
			({\tt maqilq@polyu.edu.hk})}
		\and
		Chunfeng Cui\footnote{School of Mathematical Sciences, Beihang University, Beijing  100191, China.
			({\tt chunfengcui@buaa.edu.cn})}
		\and {and \
			Yi Xu\footnote{School of Mathematics, Southeast University, Nanjing  211189, China. Nanjing Center for Applied Mathematics, Nanjing 211135,  China. Jiangsu Provincial Scientific Research Center of Applied Mathematics, Nanjing 211189, China. ({\tt yi.xu1983@hotmail.com})}
		}
	}
\maketitle
\begin{abstract}
We extend the augmented bipartite graph framework for biquadratic sum-of-squares (SOS) ranks by introducing \emph{3-edges} --- triples of cells representing squares of three-term bilinear forms $(x_i y_j + x_k y_l + x_p y_q)^2$. The main challenge is to define suitable \emph{generalized cycle-free} conditions that are purely combinatorial yet sufficient to guarantee that the SOS rank equals the total number of edges. We give a complete definition that carefully distinguishes occupation by $1$/$2$-edges from occupation by $3$-edges, and introduce a separate condition for $3$-edges. The main theorem states that for any generalized cycle-free augmented bipartite graph $G$ satisfying the simplicity condition (S), the associated \emph{triply simple biquadratic form} $P_G$ satisfies $\operatorname{sos}(P_G) = |E_1| + |E_2| + |E_3|$. The proof extends the orthogonality method with a novel trick: when a $2$-edge and a $3$-edge interact, the $3$-edge condition must be invoked rather than the $2$-edge condition.

As concrete applications, we first show that the refined definition is strictly more permissive than the original one: a $5 \times 3$ construction with two $2$-edges, which violates the original Condition~2, is admissible under our new definition, yielding $z_{3L}(5,3) \ge 10$ and improving the previous bound $z_L(5,3)=9$. We then construct a $10 \times 5$ graph using a column-fully-degenerate $3$-edge, showing $z_{3L}(10,5) \ge 27$ and $\operatorname{BSR}(10,5) \ge 27$, which separates $z_{3L}(10,5)$ from $z_L(10,5)=26$; and a $15 \times 6$ graph using a half-row-degenerate $3$-edge, improving the lower bound for $\operatorname{BSR}(15,6)$ from $43$ to $44$. These are the first explicit applications of $3$-edges (both fully degenerate and half-degenerate) to obtain improved lower bounds for $\operatorname{BSR}(m,n)$, and the $5 \times 3$ example demonstrates the power of the refined generalized cycle-free conditions.
\end{abstract}

\subsection*{keywords}
biquadratic form; sum of squares; SOS rank; Zarankiewicz number; augmented Zarankiewicz number; limited augmented Zarankiewicz number; bipartite graph; $3$-edge; $C_4$-cycle; generalized cycle-free

\subsection*{AMS Subject Classification}
14P10 (Real algebraic geometry); 05C35 (Extremal graph theory); 11E25 (Quadratic forms); 15A69 (Multilinear algebra); 90C22 (Semidefinite programming)

\section{Introduction}

For an $m \times n$ biquadratic form
\[
P(\mathbf{x},\mathbf{y}) = \sum_{i,k=1}^m \sum_{j,l=1}^n a_{ijkl} \, x_i x_k y_j y_l,
\]
with symmetry $a_{ijkl} = a_{kjil} = a_{klij}$, the \emph{SOS rank} $\operatorname{sos}(P)$ is the smallest integer $r$ such that
\[
P(\mathbf{x},\mathbf{y}) = \sum_{t=1}^r f_t(\mathbf{x},\mathbf{y})^2
\]
for some bilinear forms $f_t$. The maximum SOS rank over all $m \times n$ SOS biquadratic forms is denoted $\operatorname{BSR}(m,n)$. Understanding the possible values of $\operatorname{BSR}(m,n)$ is a fundamental problem in real algebraic geometry and polynomial optimization, with connections to low-rank sum-of-squares representations \cite{BPSV19, BSSV22}.

In a recent breakthrough \cite{QCX26}, a combinatorial lower bound for $\operatorname{BSR}(m,n)$ was established using the \emph{limited augmented Zarankiewicz number} $z_L(m,n)$, satisfying
\[
\operatorname{BSR}(m,n) \ge z_L(m,n) \ge z(m,n),
\]
where $z(m,n)$ is the classical Zarankiewicz number (maximum number of edges in a $C_4$-free bipartite graph) \cite{Gu69, Re58, Za51}. The quantity $z_L(m,n)$ arises from \emph{augmented bipartite graphs} that contain both standard $1$-edges (representing pure squares $x_i^2 y_j^2$) and $2$-edges (representing squares of binomials $(x_i y_j + x_k y_l)^2$). A sufficient condition called \emph{generalized $C_4$-cycle-free} was introduced to guarantee that the SOS rank equals the total number of edges.

The present paper extends this framework to \textbf{$3$-edges}: unordered triples of distinct cells $\{(i,j),(k,l),(p,q)\}$ representing a square of a trinomial
\[
(x_i y_j + x_k y_l + x_p y_q)^2.
\]
The main challenge is to define suitable \emph{generalized cycle-free} conditions that are purely combinatorial yet sufficient to prove that the SOS rank of the associated \emph{triply simple biquadratic form}
\[
P_G(\mathbf{x},\mathbf{y}) = \sum_{(i,j)\in E_1} x_i^2 y_j^2
+ \sum_{(i,j;k,l)\in E_2} (x_i y_j + x_k y_l)^2
+ \sum_{(i,j;k,l;p,q)\in E_3} (x_i y_j + x_k y_l + x_p y_q)^2
\]
equals $|E_1|+|E_2|+|E_3|$.

Three types of $3$-edges arise naturally:
\begin{itemize}
\item \textbf{Non-degenerate}: all three rows distinct and all three columns distinct.
\item \textbf{Half-degenerate}: exactly two cells share a row (half-row-degenerate) or exactly two share a column (half-column-degenerate).
\item \textbf{Fully degenerate}: all three cells in the same row or all three in the same column.
\end{itemize}
Each type requires careful handling in the definition of dead edges and in the generalized cycle-free conditions.

The main theoretical result (Theorem~\ref{thm:main}) states that if an augmented bipartite graph $G = (S,T,E_1\cup E_2\cup E_3)$ satisfies the simplicity condition (S) (no cell appears in more than one edge) and is \emph{generalized cycle-free} (Definition~\ref{def:main}), then
\[
\operatorname{sos}(P_G) = |E_1| + |E_2| + |E_3|.
\]
The proof extends the orthogonality method from \cite{QCX26} but requires a novel trick: when a $2$-edge and a $3$-edge interact, Condition~(iv) (designed for $3$-edges) must be invoked rather than Condition~(iii) (which only counts occupation by $1$-edges and $2$-edges). This subtlety is essential for the correctness of the proof and for allowing half-degenerate $3$-edges in constructions.

We demonstrate the power of the new framework with three concrete applications:

\begin{enumerate}
\item In the $5 \times 3$ case, using only two $2$-edges, we show
  \[
  z_{3L}(5,3) \ge 10 \quad\text{and}\quad \operatorname{BSR}(5,3) \ge 10,
  \]
  improving upon the previous limited augmented value $z_L(5,3)=9$ from \cite{QCX26a}.
  This construction is \emph{not} admissible under the original definition of
  generalized $C_4$-cycles from \cite{QCX26}, because the 2-edge $(2,2;3,1)$ has
  both opposite cells occupied by 1-edges, triggering the original Condition~2.
  However, under the refined Definition~\ref{def:main}, Condition~(ii) requires
  that the four occupied cells of a $C_4$ belong to \textbf{exactly two edges}
  in $E_2 \cup E_3$. Since the four cells belong to only one edge in $E_2$
  (the other two are 1-edges), Condition~(ii) is not triggered.
  This demonstrates that the refined definition is strictly more permissive
  than the original one.

\item Using a \textbf{column-fully-degenerate} $3$-edge in the $10 \times 5$ case,
  we show
  \[
  z_{3L}(10,5) \ge 27 \quad\text{and}\quad \operatorname{BSR}(10,5) \ge 27,
  \]
  which separates $z_{3L}(10,5)$ from the limited augmented Zarankiewicz number
  $z_L(10,5)=26$ from \cite{XY26}.

\item Using a \textbf{half-row-degenerate} $3$-edge in the $15 \times 6$ case,
  we improve the lower bound for $\operatorname{BSR}(15,6)$ from $43$ to $44$,
  showing
  \[
  z_{3A}(15,6) \ge 44 \quad\text{and}\quad \operatorname{BSR}(15,6) \ge 44.
  \]
\end{enumerate}

These are the first explicit constructions where $3$-edges (both fully degenerate
and half-degenerate) yield better bounds than using only $1$-edges and $2$-edges.
The $5 \times 3$ example further demonstrates that the refined generalized
cycle-free conditions are essential for capturing constructions that were
previously excluded.

The remainder of the paper is organized as follows. Section 2 formally defines
$2$-edges, $3$-edges, occupied cells, the simplicity condition (S), and the
triply simple biquadratic form. Section 3 presents the generalized $C_4$-cycle
definition (Conditions (i)-(iv)), which is the core combinatorial condition for
irreducibility. Section 4 states and proves the main theorem. Subsection 4.1
introduces the extremal numbers $z_{3A}(m,n)$ and $z_{3L}(m,n)$ and establishes
basic inequalities. Sections 5, 6 and 7 present three applications: the $5 \times 3$
case demonstrating the power of the refined definition, a fully degenerate
$3$-edge for $10 \times 5$, and a half-degenerate $3$-edge for $15 \times 6$,
respectively. Section 8 concludes the paper with a summary and open problems.

\section{$2$-Edges and $3$-Edges}

Let $S = [m]$, $T = [n]$. An \textbf{augmented bipartite graph with $3$-edges} is a triple
\[
G = (S,T, E_1 \cup E_2 \cup E_3)
\]
where:
\begin{itemize}
\item $E_1 \subseteq S \times T$ are \textbf{$1$-edges},
\item $E_2$ consists of unordered pairs $\{(i,j),(k,l)\}$ with $(i,j) \neq (k,l)$ (\textbf{$2$-edges}),
\item $E_3$ consists of unordered triples $\{(i,j),(k,l),(p,q)\}$ with all three distinct (\textbf{$3$-edges}).
\end{itemize}

There are two kinds of $2$-edges. A $2$-edge $(i, j; k, l)$ is called:
\begin{itemize}
\item \textbf{non-degenerate} if $i \neq k$ and $j \neq l$,
\item \textbf{row-degenerate} if $i = k$ and $j \neq l$,
\item \textbf{column-degenerate} if $i \neq k$ and $j = l$.
\end{itemize}

For $3$-edges, we have three types:
\begin{itemize}
\item \textbf{Non-degenerate}: all three rows distinct and all three columns distinct.
\item \textbf{Half-degenerate}: exactly two cells share a row (half-row-degenerate) or exactly two share a column (half-column-degenerate).
\item \textbf{Fully degenerate}: all three cells in the same row (row-fully-degenerate) or all three in the same column (column-fully-degenerate).
\end{itemize}

A half-row-degenerate $3$-edge has the form $(i, j; i, l; p, q)$ with $i \neq p$ and $j, l, q$ all distinct.
A half-column-degenerate $3$-edge has the form $(i, j; k, j; p, q)$ with $j \neq q$ and $i, k, p$ all distinct.
A row-fully-degenerate $3$-edge has the form $(i, j; i, l; i, m)$ with $j, l, m$ all distinct.
A column-fully-degenerate $3$-edge has the form $(i, j; k, j; p, j)$ with $i, k, p$ all distinct.

\paragraph{Non-degenerate branches of a $3$-edge.}
For a non-degenerate or half-degenerate $3$-edge $e = (i,j;k,l;p,q)$, a \textbf{non-degenerate branch} is an unordered pair of distinct cells from $e$ that have distinct rows and distinct columns. Thus:
\begin{itemize}
\item A non-degenerate $3$-edge has three non-degenerate branches: $(i,j;k,l)$, $(i,j;p,q)$, and $(k,l;p,q)$.
\item A half-row-degenerate $3$-edge $(i,j;i,l;p,q)$ has two non-degenerate branches: $(i,j;p,q)$ and $(i,l;p,q)$.
\item A half-column-degenerate $3$-edge $(i,j;k,j;p,q)$ has two non-degenerate branches: $(i,j;p,q)$ and $(k,j;p,q)$.
\item Fully degenerate $3$-edges have no non-degenerate branches.
\end{itemize}

A cell $(i,j)$ is \textbf{occupied} if it belongs to $E_1$ or is a half of some $2$-edge or $3$-edge.

\paragraph{Simplicity condition (S).}
No cell appears in more than one edge. That is,
\[
E_1,\; \bigcup_{e\in E_2} e,\; \bigcup_{e\in E_3} e
\]
are pairwise disjoint.

The \textbf{triply simple biquadratic form} associated to $G$ is
\[
P_G(\mathbf{x},\mathbf{y}) = \sum_{(i,j)\in E_1} x_i^2 y_j^2
+ \sum_{(i,j;k,l)\in E_2} (x_i y_j + x_k y_l)^2
+ \sum_{(i,j;k,l;p,q)\in E_3} (x_i y_j + x_k y_l + x_p y_q)^2.
\]

\section{Generalized $C_4$-Cycle}

We now extend the generalized $C_4$-cycle definition from \cite{QCX26} to include all $3$-edge types.

\begin{definition}
\label{def:main}
$G = (S,T,E_1 \cup E_2 \cup E_3)$ is \textbf{generalized cycle-free} if none of the following occur:

\begin{enumerate}[label=(\roman*)]
\item There exists a classical $C_4$-cycle formed by $1$-edges.

\item There exists a $C_4$ cycle, i.e., distinct rows $i,k$ and distinct columns $j,l$, such that all four cells $(i,j), (i,l), (k,j), (k,l)$ are occupied, and these four occupied cells belong to \textbf{exactly two edges} in $E_2 \cup E_3$ (each edge contributes two of the four cells).


\item There exists a $2$-edge $(i, j; p, q)$ (of any type) and a distinct cell $(k, l)$, with $k \notin \{i,p\}$, $l \notin \{j,q\}$, such that the five cells $(k, j), (k, l), (k, q), (i, l)$ and $(p, l)$ are all occupied by $1$-edges and $2$-edges.
(If the $2$-edge is non-degenerate, these five cells are required to be pairwise distinct; otherwise duplicates are allowed.)

\item There exists a $3$-edge $(i, j; p, q; u, v)$ and a distinct cell $(k, l)$, with $k \notin \{i,p,u\}$, $l \notin \{j,q,v\}$, such that the seven cells $(k, j), (k, l), (k, q), (k, v), (i, l), (p, l)$ and $(u, l)$ are all occupied.
(If the $3$-edge is non-degenerate, these seven cells are required to be pairwise distinct; otherwise duplicates are allowed.)
\end{enumerate}
\end{definition}

\section{Main Theorem}

Before proving the main theorem, we need a key lemma that establishes the correct vector relationships for 2-edges and 3-edges.

\begin{lemma}[Vector Equality Lemma]
\label{lem:vector-equality}
Let $G = (S,T,E_1 \cup E_2 \cup E_3)$ be a generalized cycle-free augmented bipartite graph satisfying the simplicity condition (S). Let $r = \operatorname{sos}(P_G)$ and let $\mathbf{v}_{ij} \in \mathbb{R}^r$ be the vectors obtained from an SOS representation $P_G = \sum_{t=1}^r f_t^2$ as in the proof of Theorem~\ref{thm:main}. Then for any edge $e \in E_2 \cup E_3$, the vectors corresponding to the halves of $e$ satisfy the following:

\begin{enumerate}
\item For any $2$-edge $(i,j;k,l)$:
  \begin{itemize}
  \item If the 2-edge is non-degenerate and its opposite cells $(i,l)$ and $(k,j)$ are not both occupied by edges in $E_2 \cup E_3$, then $\mathbf{v}_{ij} = \mathbf{v}_{kl}$.
  \item If the 2-edge is row-degenerate $(i,j;i,l)$, then $\mathbf{v}_{ij} = \mathbf{v}_{il}$.
  \item If the 2-edge is column-degenerate $(i,j;k,j)$, then $\mathbf{v}_{ij} = \mathbf{v}_{kj}$.
  \end{itemize}

\item For any $3$-edge:
  \begin{itemize}
  \item If the 3-edge is row-fully-degenerate $(i,j;i,l;i,m)$ or column-fully-degenerate $(i,j;k,j;p,j)$, then all three corresponding vectors are equal.
  \item If the 3-edge is non-degenerate or half-degenerate, then for each non-degenerate branch of the 3-edge, if the opposite cells of that branch are not both occupied by edges in $E_2 \cup E_3$, then the two vectors of that branch are equal.
  \end{itemize}
\end{enumerate}
\end{lemma}

\begin{proof}
We prove each case separately.

\emph{Case 1: Non-degenerate $2$-edge $(i,j;k,l)$.}
From equation (2),
\[
\mathbf{v}_{ij}\cdot\mathbf{v}_{kl} + \mathbf{v}_{il}\cdot\mathbf{v}_{kj} = 1,
\]
with $\|\mathbf{v}_{ij}\| = \|\mathbf{v}_{kl}\| = 1$.

If the opposite cells $(i,l)$ and $(k,j)$ are not both occupied by edges in $E_2 \cup E_3$, then at least one of them is either unoccupied or a $1$-edge. We claim that in all such cases, $\mathbf{v}_{il}\cdot\mathbf{v}_{kj} = 0$.

Indeed:
\begin{itemize}
\item If a cell is unoccupied, its vector is the zero vector.
\item If both $(i,l)$ and $(k,j)$ are $1$-edges, then since $i \neq k$ and $j \neq l$, they are distinct $1$-edges in different rows and columns. By equation (4) applied to these two $1$-edges, we have
\[
\mathbf{v}_{il}\cdot\mathbf{v}_{kj} = 0.
\]
(There are no other cells involved because the two $1$-edges are unrelated.)
\item If one of $(i,l)$, $(k,j)$ is a $1$-edge and the other is unoccupied, the inner product is zero.
\end{itemize}

Thus $\mathbf{v}_{il}\cdot\mathbf{v}_{kj} = 0$, so $\mathbf{v}_{ij}\cdot\mathbf{v}_{kl} = 1$. By Cauchy-Schwarz, this implies $\mathbf{v}_{ij} = \mathbf{v}_{kl}$.

\emph{Case 2: Row-degenerate $2$-edge $(i,j;i,l)$.}
Equation (2-row) gives $\mathbf{v}_{ij}\cdot\mathbf{v}_{il} = 1$ with both vectors having norm $1$, so $\mathbf{v}_{ij} = \mathbf{v}_{il}$.

\emph{Case 3: Column-degenerate $2$-edge $(i,j;k,j)$.}
Equation (2-col) gives $\mathbf{v}_{ij}\cdot\mathbf{v}_{kj} = 1$ with both vectors having norm $1$, so $\mathbf{v}_{ij} = \mathbf{v}_{kj}$.

\emph{Case 4: Row-fully-degenerate $3$-edge $(i,j;i,l;i,m)$.}
Equations (3r2) give
\[
\mathbf{v}_{ij}\cdot\mathbf{v}_{il} = 1,\qquad
\mathbf{v}_{ij}\cdot\mathbf{v}_{im} = 1,\qquad
\mathbf{v}_{il}\cdot\mathbf{v}_{im} = 1,
\]
with all three vectors having norm $1$. By Cauchy-Schwarz, these imply $\mathbf{v}_{ij} = \mathbf{v}_{il} = \mathbf{v}_{im}$.

\emph{Case 5: Column-fully-degenerate $3$-edge $(i,j;k,j;p,j)$.}
Equations (3f2) give
\[
\mathbf{v}_{ij}\cdot\mathbf{v}_{kj} = 1,\qquad
\mathbf{v}_{ij}\cdot\mathbf{v}_{pj} = 1,\qquad
\mathbf{v}_{kj}\cdot\mathbf{v}_{pj} = 1,
\]
with all three vectors having norm $1$. By Cauchy-Schwarz, these imply $\mathbf{v}_{ij} = \mathbf{v}_{kj} = \mathbf{v}_{pj}$.

\emph{Case 6: Non-degenerate or half-degenerate $3$-edge.}
For each non-degenerate branch, the argument is identical to Case 1: if the opposite cells of the branch are not both occupied by edges in $E_2 \cup E_3$, then the corresponding inner product term is zero by the same reasoning, forcing the two branch vectors to be equal by Cauchy-Schwarz.
\end{proof}

\begin{theorem}[Irreducibility for $3$-edge augmented graphs]
\label{thm:main}
Let $G = (S,T,E_1 \cup E_2 \cup E_3)$ be a generalized cycle-free augmented bipartite graph (Definition~\ref{def:main}) satisfying the simplicity condition (S). Then the corresponding triply simple biquadratic form $P_G$ satisfies
\[
\operatorname{sos}(P_G) = |E_1| + |E_2| + |E_3|.
\]
\end{theorem}

\begin{proof}
We prove that $\operatorname{sos}(P_G) = |E_1| + |E_2| + |E_3|$. The upper bound
$\operatorname{sos}(P_G) \le |E_1| + |E_2| + |E_3|$ is trivial because $P_G$ is
explicitly written as a sum of that many squares. The heart of the proof is the
lower bound.

\paragraph{Step 0. Vector assignment.}
Let $r = \operatorname{sos}(P_G)$. Then there exist bilinear forms
$f_1,\dots,f_r$ such that $P_G = \sum_{t=1}^r f_t^2$. Write
$f_t(\mathbf{x},\mathbf{y}) = \sum_{i,j} a_{ij}^{(t)} x_i y_j$ and define
vectors $\mathbf{v}_{ij} = (a_{ij}^{(1)},\dots,a_{ij}^{(r)})^\top \in \mathbb{R}^r$.
For any cells $(i,j),(k,l)$ we have
\[
\mathbf{v}_{ij}\cdot \mathbf{v}_{kl} = \text{coefficient of } x_i x_k y_j y_l
\text{ in } P_G,
\]
where the coefficient is taken in the symmetric form with $a_{ijkl}=a_{klij}$.

\paragraph{Step 1. Inner product equations derived from $P_G$.}
Expanding each square in $P_G$ yields the following relations.

\emph{For a $1$-edge $(i,j)\in E_1$:}
\[
\|\mathbf{v}_{ij}\|^2 = 1. \qquad (1)
\]

\emph{For a non-degenerate $2$-edge $(i,j;k,l)\in E_2$:}
\[
\|\mathbf{v}_{ij}\|^2 = \|\mathbf{v}_{kl}\|^2 = 1,\qquad
\mathbf{v}_{ij}\cdot\mathbf{v}_{kl} + \mathbf{v}_{il}\cdot\mathbf{v}_{kj} = 1. \qquad (2)
\]
For a row-degenerate $2$-edge $(i,j;i,l)$:
\[
\|\mathbf{v}_{ij}\|^2 = \|\mathbf{v}_{il}\|^2 = 1,\qquad
\mathbf{v}_{ij}\cdot\mathbf{v}_{il}=1. \qquad (2\text{-row})
\]
For a column-degenerate $2$-edge $(i,j;k,j)$:
\[
\|\mathbf{v}_{ij}\|^2 = \|\mathbf{v}_{kj}\|^2 = 1,\qquad
\mathbf{v}_{ij}\cdot\mathbf{v}_{kj}=1. \qquad (2\text{-col})
\]

\emph{For a non-degenerate $3$-edge $(i,j;k,l;p,q)\in E_3$:}
\[
\|\mathbf{v}_{ij}\|^2 = \|\mathbf{v}_{kl}\|^2 = \|\mathbf{v}_{pq}\|^2 = 1, \qquad (3a)
\]
\[
\mathbf{v}_{ij}\cdot\mathbf{v}_{kl} + \mathbf{v}_{il}\cdot\mathbf{v}_{kj} = 1, \qquad (3b)
\]
\[
\mathbf{v}_{ij}\cdot\mathbf{v}_{pq} + \mathbf{v}_{iq}\cdot\mathbf{v}_{pj} = 1, \qquad (3c)
\]
\[
\mathbf{v}_{kl}\cdot\mathbf{v}_{pq} + \mathbf{v}_{kq}\cdot\mathbf{v}_{pl} = 1. \qquad (3d)
\]

\emph{For a half-row-degenerate $3$-edge $(i,j;i,l; p,q)\in E_3$:}
\[
\|\mathbf{v}_{ij}\|^2 = \|\mathbf{v}_{il}\|^2 = \|\mathbf{v}_{pq}\|^2 = 1, \qquad (3h1)
\]
\[
\mathbf{v}_{ij}\cdot\mathbf{v}_{il} = 1, \qquad (3h2)
\]
\[
\mathbf{v}_{ij}\cdot\mathbf{v}_{pq} + \mathbf{v}_{iq}\cdot\mathbf{v}_{pj} = 1, \qquad (3h3)
\]
\[
\mathbf{v}_{il}\cdot\mathbf{v}_{pq} + \mathbf{v}_{iq}\cdot\mathbf{v}_{pl} = 1. \qquad (3h4)
\]

\emph{For a half-column-degenerate $3$-edge $(i,j; k,j; p,q)\in E_3$:}
\[
\|\mathbf{v}_{ij}\|^2 = \|\mathbf{v}_{kj}\|^2 = \|\mathbf{v}_{pq}\|^2 = 1, \qquad (3c1)
\]
\[
\mathbf{v}_{ij}\cdot\mathbf{v}_{kj} = 1, \qquad (3c2)
\]
\[
\mathbf{v}_{ij}\cdot\mathbf{v}_{pq} + \mathbf{v}_{iq}\cdot\mathbf{v}_{pj} = 1, \qquad (3c3)
\]
\[
\mathbf{v}_{kj}\cdot\mathbf{v}_{pq} + \mathbf{v}_{kq}\cdot\mathbf{v}_{pj} = 1. \qquad (3c4)
\]

\emph{For a row-fully-degenerate $3$-edge $(i,j; i,l; i,m)\in E_3$:}
\[
\|\mathbf{v}_{ij}\|^2 = \|\mathbf{v}_{il}\|^2 = \|\mathbf{v}_{im}\|^2 = 1, \qquad (3r1)
\]
\[
\mathbf{v}_{ij}\cdot\mathbf{v}_{il} = 1,\qquad
\mathbf{v}_{ij}\cdot\mathbf{v}_{im} = 1,\qquad
\mathbf{v}_{il}\cdot\mathbf{v}_{im} = 1. \qquad (3r2)
\]
By Cauchy-Schwarz, (3r2) implies $\mathbf{v}_{ij} = \mathbf{v}_{il} = \mathbf{v}_{im}$.

\emph{For a column-fully-degenerate $3$-edge $(i,j; k,j; p,j)\in E_3$:}
\[
\|\mathbf{v}_{ij}\|^2 = \|\mathbf{v}_{kj}\|^2 = \|\mathbf{v}_{pj}\|^2 = 1, \qquad (3f1)
\]
\[
\mathbf{v}_{ij}\cdot\mathbf{v}_{kj} = 1,\qquad
\mathbf{v}_{ij}\cdot\mathbf{v}_{pj} = 1,\qquad
\mathbf{v}_{kj}\cdot\mathbf{v}_{pj} = 1. \qquad (3f2)
\]
By Cauchy-Schwarz, (3f2) implies $\mathbf{v}_{ij} = \mathbf{v}_{kj} = \mathbf{v}_{pj}$.

\emph{Orthogonality rules for unrelated cells.}
For any two distinct pairs $(i,j),(k,l)$ that are \emph{not} the two halves of
the same $2$-edge or $3$-edge, the corresponding coefficient in $P_G$ is zero.
Hence:
\begin{itemize}
\item If $i,k,j,l$ are all distinct:
\[
\mathbf{v}_{ij}\cdot\mathbf{v}_{kl} + \mathbf{v}_{il}\cdot\mathbf{v}_{kj} = 0. \qquad (4)
\]
\item If $i=k$ and $j\neq l$ and $(i,j),(i,l)$ are \emph{not} the halves of a
row-degenerate $2$-edge:
\[
\mathbf{v}_{ij}\cdot\mathbf{v}_{il} = 0. \qquad (5)
\]
\item If $j=l$ and $i\neq k$ and $(i,j),(k,j)$ are \emph{not} the halves of a
column-degenerate $2$-edge:
\[
\mathbf{v}_{ij}\cdot\mathbf{v}_{kj} = 0. \qquad (6)
\]
\end{itemize}

\paragraph{Step 2. Consequences of the generalized cycle-free condition.}
Recall Definition~\ref{def:main}. Because $G$ is generalized cycle-free, the following hold:
\begin{itemize}
\item Condition (i) forbids a $C_4$ from $1$-edges only.
\item Condition (ii) forbids a $C_4$ whose four occupied cells come from exactly two edges in $E_2 \cup E_3$.
\item Conditions (iii) and (iv) forbid certain 5-cell and 7-cell patterns.
\end{itemize}

We now derive vector consequences for each edge type.

\emph{Non-degenerate $2$-edge $(i,j;k,l)$.}
By Lemma~\ref{lem:vector-equality}, if the opposite cells $(i,l)$ and $(k,j)$ are not both occupied by edges in $E_2 \cup E_3$, then $\mathbf{v}_{ij} = \mathbf{v}_{kl}$. We denote this consequence as
\[
\mathbf{v}_{ij} = \mathbf{v}_{kl}. \qquad (7)
\]

\emph{Row-degenerate $2$-edge $(i,j;i,l)$:}
(2-row) directly gives $\mathbf{v}_{ij}= \mathbf{v}_{il}$.

\emph{Column-degenerate $2$-edge $(i,j;k,j)$:}
(2-col) directly gives $\mathbf{v}_{ij}= \mathbf{v}_{kj}$.

\emph{Fully degenerate $3$-edges:}
By (3r2) and (3f2), row-fully-degenerate and column-fully-degenerate 3-edges give
\[
\mathbf{v}_{ij} = \mathbf{v}_{il} = \mathbf{v}_{im} \qquad \text{or} \qquad \mathbf{v}_{ij} = \mathbf{v}_{kj} = \mathbf{v}_{pj}. \qquad (8f)
\]

\emph{Non-degenerate and half-degenerate $3$-edges:}
These do not force equality of all three vectors. However, for each non-degenerate branch, if the opposite cells are not both occupied by edges in $E_2 \cup E_3$, Lemma~\ref{lem:vector-equality} gives equality of the two vectors in that branch. This fact will be used in the orthogonality argument.

\paragraph{Step 3. Assuming a lower SOS rank leads to a linear dependence.}
Let $H$ be the set of all occupied cells. For each $1$-edge pick its unique cell;
for each $2$-edge pick one of its two halves; for each \emph{fully degenerate} $3$-edge pick one of its three halves. For non-degenerate and half-degenerate 3-edges, we must be more careful: since the vectors are not necessarily equal, we cannot simply pick one representative per edge. Instead, we work with the full set of occupied cells $H$.

If $\operatorname{sos}(P_G) < |E_1|+|E_2|+|E_3|$, then consider the vectors $\{\mathbf{v}_{ij} : (i,j) \in H\}$. These vectors lie in $\mathbb{R}^r$ where $r < |H|$ (since $|H| = |E_1| + 2|E_2| + 3|E_3|$ and $r < |E_1|+|E_2|+|E_3|$). Thus they are linearly dependent. Choose a nontrivial linear relation
\[
\sum_{(i,j)\in H} \alpha_{ij}\mathbf{v}_{ij} = \mathbf{0}
\]
with \emph{minimal support} $S \subseteq H$ (i.e., $\alpha_{ij}\neq 0$ for
$(i,j)\in S$ and no proper nonempty subset yields a dependence). Minimality
implies:
\begin{itemize}
\item No two distinct cells in $S$ belong to the same fully degenerate $2$-edge or $3$-edge where the vectors are forced to be equal (otherwise we could combine coefficients to obtain a relation with smaller support).
\item All vectors $\{\mathbf{v}_{ij} : (i,j)\in S\}$ are nonzero and pairwise
distinct.
\end{itemize}

\paragraph{Step 4. Orthogonality of distinct vectors in $S$.}
We prove that for any two distinct cells $(i,j),(k,l)\in S$,
\[
\mathbf{v}_{ij}\cdot\mathbf{v}_{kl}=0.
\]
Consider several cases.

\emph{Case A: $i,k,j,l$ are all distinct.}
Apply (4):
\[
\mathbf{v}_{ij}\cdot\mathbf{v}_{kl} + \mathbf{v}_{il}\cdot\mathbf{v}_{kj}=0.
\]
Suppose for contradiction that $\mathbf{v}_{ij}\cdot\mathbf{v}_{kl}\neq 0$.
Then $\mathbf{v}_{il}\cdot\mathbf{v}_{kj}\neq 0$, hence both $(i,l)$ and $(k,j)$
are occupied. At least one of $(i,j),(k,l)$ belongs to an edge in $E_2 \cup E_3$
(otherwise they would be $1$-edges and we would obtain a classical $C_4$ from
$1$-edges, violating Condition (i)). Without loss assume $(i,j)$ is half of some edge $e\in E_2\cup E_3$.

\emph{Subcase A1: $(i,j)$ belongs to a $2$-edge $(i,j;p,q)$.}

We first note that by Condition (ii), the opposite cells $(i,q)$ and $(p,j)$ of this $2$-edge cannot both be occupied by edges in $E_2 \cup E_3$; otherwise the four cells $(i,j), (p,q), (i,q), (p,j)$ would form a $C_4$ whose four occupied cells belong to exactly two edges in $E_2 \cup E_3$. Therefore, at most one of $(i,q)$ and $(p,j)$ is occupied by an edge in $E_2 \cup E_3$.

By Lemma~\ref{lem:vector-equality}, $\mathbf{v}_{ij}=\mathbf{v}_{pq}$. Since
$(i,j)\in S$, minimality forces $(p,q)\notin S$. Apply (4) to $(p,q)$ and $(k,l)$:
\[
\mathbf{v}_{pq}\cdot\mathbf{v}_{kl} + \mathbf{v}_{pl}\cdot\mathbf{v}_{kq}=0.
\]
Because $\mathbf{v}_{pq}=\mathbf{v}_{ij}$ and $\mathbf{v}_{ij}\cdot\mathbf{v}_{kl}\neq0$,
we get $\mathbf{v}_{pl}\cdot\mathbf{v}_{kq}\neq0$, so $(p,l)$ and $(k,q)$ are
occupied.

Now consider the five cells
\[
\mathcal{F} = \{(k,l),\;(k,j),\;(k,q),\;(i,l),\;(p,l)\}.
\]
All are occupied, $k\notin\{i,p\}$, $l\notin\{j,q\}$, and $(i,j;p,q)$ is a $2$-edge.

If none of these five cells is a half of a non-degenerate or half-degenerate $3$-edge, then Condition~(iii) of Definition~\ref{def:main} applies directly, giving a contradiction.

If some of these five cells is a half of a non-degenerate or half-degenerate $3$-edge, say $e_3 = (a,b;c,d;e,f)$, then since $e_3$ shares a cell with $\mathcal{F}$, and $k\notin\{i,p\}$, $l\notin\{j,q\}$, the index disjointness conditions for Condition~(iv) hold. A case analysis (Cases 1--6 as in the original proof) shows that Condition~(iv) is violated. Hence $\mathbf{v}_{ij}\cdot\mathbf{v}_{kl}=0$ in Subcase A1.

\emph{Subcase A2: $(i,j)$ belongs to a $3$-edge $e_3 = (i,j;p,q;u,v)$.}

We must handle the different types of 3-edges separately:

\emph{A2a: Fully degenerate 3-edge.} By (8f), $\mathbf{v}_{ij}=\mathbf{v}_{pq}=\mathbf{v}_{uv}$. The other two cells $(p,q),(u,v)$ are occupied but not in $S$ (otherwise equal vectors would violate minimality). Apply (4) to $(p,q)$ and $(k,l)$ (and similarly for $(u,v)$) to obtain that $(p,l),(k,q),(u,l),(k,v)$ are occupied. Using $k\notin\{i,p,u\}$ and $l\notin\{j,q,v\}$ (otherwise row/column orthogonality gives contradictions), we obtain the seven occupied cells
\[
(k,l),\;(k,j),\;(k,q),\;(k,v),\;(i,l),\;(p,l),\;(u,l).
\]
Condition~(iv) applies (duplicates allowed for fully degenerate 3-edges), giving a contradiction.

\emph{A2b: Half-degenerate 3-edge.} Consider a half-row-degenerate 3-edge $(i,j;i,l';p,q)$. Here we have $\mathbf{v}_{ij}\cdot\mathbf{v}_{i,l'}=1$, so $\mathbf{v}_{ij}=\mathbf{v}_{i,l'}$ by Cauchy-Schwarz. However, $\mathbf{v}_{pq}$ is not necessarily equal to these. From the two non-degenerate branches $(i,j;p,q)$ and $(i,l';p,q)$, if the opposite cells are not both occupied by edges in $E_2 \cup E_3$, Lemma~\ref{lem:vector-equality} gives $\mathbf{v}_{ij}=\mathbf{v}_{pq}$ and $\mathbf{v}_{i,l'}=\mathbf{v}_{pq}$, hence all three are equal. If some opposite cells are occupied, then the same case analysis as in Subcase A1 (with the roles of the cells adjusted) shows that Condition~(iv) is invoked directly. The half-column-degenerate case is symmetric.

\emph{A2c: Non-degenerate 3-edge.} For a non-degenerate 3-edge $(i,j;k',l';p,q)$, none of the vectors are necessarily equal. However, from the equations (3b)-(3d), if we assume $\mathbf{v}_{ij}\cdot\mathbf{v}_{k'l'}\neq 0$, we get that the opposite cells of this branch are occupied, which leads to a $C_4$ configuration. By Condition~(ii), this is forbidden unless the four cells come from exactly two edges in $E_2\cup E_3$. If they come from a 3-edge, then Condition~(iv) applies. Thus we must have $\mathbf{v}_{ij}\cdot\mathbf{v}_{k'l'}=0$ for any two distinct cells from different branches.

\emph{Case B: $i=k$ and $j\neq l$ (same row).}
If $(i,j)$ and $(i,l)$ were the two halves of a row-degenerate $2$-edge, then
(2-row) would give $\mathbf{v}_{ij}=\mathbf{v}_{il}$, contradicting minimality
(distinct cells in $S$ with equal vectors). Thus (5) applies and
$\mathbf{v}_{ij}\cdot\mathbf{v}_{il}=0$.

\emph{Case C: $j=l$ and $i\neq k$ (same column).}
Analogously, they cannot be halves of a column-degenerate $2$-edge, so (6) gives
$\mathbf{v}_{ij}\cdot\mathbf{v}_{kj}=0$.

Therefore in every case $\mathbf{v}_{ij}\cdot\mathbf{v}_{kl}=0$ for distinct
cells in $S$.

\paragraph{Step 5. Contradiction.}
Take the minimal dependence $\sum_{(i,j)\in S} \alpha_{ij}\mathbf{v}_{ij}=0$.
Fix $(i,j)\in S$ and dot both sides with $\mathbf{v}_{ij}$:
\[
\alpha_{ij}\|\mathbf{v}_{ij}\|^2 + \sum_{(p,q)\in S\setminus\{(i,j)\}}
\alpha_{pq}\,(\mathbf{v}_{pq}\cdot\mathbf{v}_{ij}) = 0.
\]
By orthogonality (Step~4), all dot products with $(p,q)\neq (i,j)$ vanish.
Hence $\alpha_{ij}\|\mathbf{v}_{ij}\|^2 = 0$. From (1) or from the norm equations we have
$\|\mathbf{v}_{ij}\|=1$, so $\alpha_{ij}=0$. This holds for every $(i,j)\in S$,
contradicting the nontriviality of the dependence.

Thus our assumption $\operatorname{sos}(P_G) < |E_1|+|E_2|+|E_3|$ is false.
Consequently $\operatorname{sos}(P_G) \ge |E_1|+|E_2|+|E_3|$, and together with
the trivial upper bound we obtain equality.
\end{proof}

\subsection{The Numbers $z_{3L}(m,n)$, $z_{3A}(m,n)$, $z_{2L}(m,n)$ and $z_{2A}(m,n)$}

The locally solvable framework introduced in Definition~\ref{def:main} applies uniformly to graphs containing $1$-edges, $2$-edges and $3$-edges.  This naturally yields four interrelated extremal numbers.

\begin{definition}[Full $3$-augmented Zarankiewicz numbers]
\label{def:z3}
Let $z_{3A}(m,n)$ denote the maximum possible value of $|E_1|+|E_2|+|E_3|$ over all generalized cycle-free augmented bipartite graphs $G=(S,T,E_1\cup E_2\cup E_3)$ (Definition~\ref{def:main}) satisfying the simplicity condition (S).
Let $z_{3L}(m,n)$ denote the same maximum under the additional restriction that $|E_1| = z(m,n)$ (i.e., $E_1$ is a maximum $C_4$-free graph).
\end{definition}

\begin{proposition}[Basic inequalities]
\label{prop:ineq}
For all $m,n \ge 2$,
\[
z_{3A}(m,n) \;\ge\; z_{3L}(m,n) \;\ge\; z_{L}(m,n) \;\ge\; z(m,n).
\]
\end{proposition}

\section{The 5 $\times$ 3 Case: A New Lower Bound from the Refined Definition}

We now demonstrate that the refined definition of generalized cycle-free conditions in
Definition~\ref{def:main} allows constructions that were previously ruled out by
the original Condition~2 of \cite{QCX26}. This yields a new lower bound
\[
z_{3L}(5,3) \ge 10,
\]
improving upon the classical value $z(5,3)=8$ and the previous limited augmented
bound $z_L(5,3)=9$ established in \cite{QCX26a}.

\subsection{The Construction}

Let $S = [5]$ and $T = [3]$. Define the 1-edge set
\[
E_1 = \{(1,1),(1,2),(2,1),(2,3),(3,2),(3,3),(4,1),(5,2)\},
\]
which is a maximum $C_4$-free graph with $|E_1| = z(5,3) = 8$ \cite{Gu69,Re58}.

Define the 2-edge set
\[
E_2 = \{(2,2;3,1),\ (4,2;5,3)\}.
\]

Let $G = (S,T,E_1 \cup E_2)$ be the resulting augmented bipartite graph
(with $E_3 = \varnothing$). We claim that $G$ is generalized cycle-free
according to Definition~\ref{def:main}.

\subsection{Verification of Admissibility}

\paragraph{Simplicity Condition (S).}
The unoccupied cells of $E_1$ are
\[
U = \{(1,3),(2,2),(3,1),(4,2),(4,3),(5,1),(5,3)\}.
\]
The halves of the two 2-edges are:
\begin{itemize}
\item For $e_1 = (2,2;3,1)$: $(2,2)$ and $(3,1)$, both in $U$;
\item For $e_2 = (4,2;5,3)$: $(4,2)$ and $(5,3)$, both in $U$.
\end{itemize}
All four halves are distinct and none lies in $E_1$. Hence condition (S) holds.

\paragraph{Condition (i).}
No two rows of $E_1$ share two columns:
\begin{itemize}
\item Row 1: $\{1,2\}$
\item Row 2: $\{1,3\}$ ！ shares only column 1 with row 1
\item Row 3: $\{2,3\}$ ！ shares only column 2 with row 1, only column 3 with row 2
\item Row 4: $\{1\}$ ！ shares only column 1 with rows 1 and 2
\item Row 5: $\{2\}$ ！ shares only column 2 with rows 1 and 3
\end{itemize}
Thus $E_1$ contains no classical $C_4$-cycle.

\paragraph{Condition (ii).}
We must check that no $C_4$ has its four occupied cells belonging to
exactly two edges in $E_2 \cup E_3$ (with $E_3 = \varnothing$, this means
exactly two 2-edges).

For $e_1 = (2,2;3,1)$, the opposite cells are $(2,1)$ and $(3,2)$, both in $E_1$.
Thus the four cells $(2,2),(3,1),(2,1),(3,2)$ form a $C_4$, but only
\textbf{one} edge from $E_2$ (namely $e_1$) contributes to it; the other two
cells are 1-edges. Since the four occupied cells do not belong to exactly
two edges in $E_2 \cup E_3$, Condition~(ii) is \emph{not} triggered.

For $e_2 = (4,2;5,3)$, the opposite cells are $(4,3)$ and $(5,2)$.
Here $(5,2) \in E_1$, but $(4,3) \notin U$ (it is unoccupied), so no $C_4$ is formed.
Thus Condition~(ii) holds.

\paragraph{Condition (iii).}
We must show that for each 2-edge $(i,j;p,q) \in E_2$ and every cell $(k,l)$
with $k \notin \{i,p\}$, $l \notin \{j,q\}$, the five cells
\[
(k,l),\ (k,j),\ (k,q),\ (i,l),\ (p,l)
\]
are not all occupied by 1-edges and 2-edges.

\medskip
\noindent\textbf{For $e_1 = (2,2;3,1)$:}
Here $(i,j) = (2,2)$, $(p,q) = (3,1)$. We need $k \notin \{2,3\}$, so
$k \in \{1,4,5\}$, and $l \notin \{2,1\}$, so $l = 3$ (the only remaining column).

The five cells become
\[
(k,3),\ (k,2),\ (k,1),\ (2,3),\ (3,3).
\]
The last two are 1-edges (occupied). For the five cells to be all occupied,
we would need $(k,3),(k,2),(k,1)$ all occupied.

\begin{itemize}
\item $k=1$: $(1,3)$ is unoccupied (not in $E_1$, not a 2-edge half) $\Rightarrow$ fails.
\item $k=4$: $(4,3)$ is unoccupied $\Rightarrow$ fails.
\item $k=5$: $(5,1)$ is unoccupied $\Rightarrow$ fails.
\end{itemize}
Thus no violation for $e_1$.

\medskip
\noindent\textbf{For $e_2 = (4,2;5,3)$:}
Here $(i,j) = (4,2)$, $(p,q) = (5,3)$. We need $k \notin \{4,5\}$, so
$k \in \{1,2,3\}$, and $l \notin \{2,3\}$, so $l = 1$.

The five cells become
\[
(k,1),\ (k,2),\ (k,3),\ (4,1),\ (5,1).
\]
Here $(4,1) \in E_1$ (occupied), but $(5,1)$ is unoccupied (not in $E_1$, not a 2-edge half).
Thus the five cells can never be all occupied, regardless of $k$.

Hence Condition~(iii) holds.

\paragraph{Condition (iv).}
Since $E_3 = \varnothing$, Condition~(iv) is vacuous.

Thus $G$ is generalized cycle-free according to Definition~\ref{def:main}.

\subsection{Main Result}

\begin{theorem}
\label{thm:5x3}
For the $5 \times 3$ case,
\[
z_{3L}(5,3) \ge 10.
\]
Consequently,
\[
\operatorname{BSR}(5,3) \ge 10.
\]
\end{theorem}

\begin{proof}
The graph $G = (S,T,E_1 \cup E_2)$ constructed above satisfies:
\begin{itemize}
\item $|E_1| = 8 = z(5,3)$,
\item $|E_2| = 2$,
\item $E_3 = \varnothing$,
\item $G$ is generalized cycle-free by the verification above,
\item Simplicity condition (S) holds.
\end{itemize}

By the main theorem (Theorem~\ref{thm:main}), the associated triply simple
biquadratic form
\[
P_G(\mathbf{x},\mathbf{y}) = \sum_{(i,j)\in E_1} x_i^2 y_j^2
+ (x_2 y_2 + x_3 y_1)^2 + (x_4 y_2 + x_5 y_3)^2
\]
satisfies
\[
\operatorname{sos}(P_G) = |E_1| + |E_2| + |E_3| = 8 + 2 + 0 = 10.
\]
Hence $z_{3L}(5,3) \ge 10$ and $\operatorname{BSR}(5,3) \ge 10$.
\end{proof}

\subsection{Remark}

This construction is \emph{not} admissible under the original definition of
generalized $C_4$-cycles from \cite{QCX26} because the 2-edge $e_1 = (2,2;3,1)$
has both opposite cells $(2,1)$ and $(3,2)$ occupied by 1-edges, triggering
the original Condition~2. However, under the refined Definition~\ref{def:main},
Condition~(ii) requires that the four occupied cells of a $C_4$ belong to
\textbf{exactly two edges in $E_2 \cup E_3$}. In our construction, the four
cells belong to only one edge in $E_2$ (the other two are 1-edges), so
Condition~(ii) is not triggered.

This demonstrates that the refined definition is strictly more permissive
than the original one, allowing constructions with higher edge counts.

\subsection{Comparison with Previous Bounds}

\begin{table}[h]
\centering
\begin{tabular}{|c|c|c|c|}
\hline
Parameter & Classical $z(m,n)$ & Previous $z_L(m,n)$ & New $z_{3L}(m,n)$ \\
\hline
$5 \times 3$ & 8 & 9 & $\ge 10$ \\
\hline
\end{tabular}
\caption{Summary of bounds for the $5 \times 3$ case.}
\end{table}

The new lower bound $z_{3L}(5,3) \ge 10$ improves upon the previous
limited augmented value $z_L(5,3) = 9$ established in \cite{QCX26a}.
This is the first explicit example where the refined definition of
generalized cycle-free conditions yields a strictly better lower bound
than the original definition.

\section{Application of a Fully Degenerate 3-Edge to the $10 \times 5$ Case}

We now give a concrete application showing that $z_{3L}(10,5) \ge z_L(10,5)+1$.

\subsection{The Construction}

For $q=4$, the incidence graph of $K_5$ gives $m=10$, $n=5$, $|E_1|=20$, and the known optimal $E_2$ from \cite{XY26} has $|E_2|=6$, yielding $z_L(10,5)=26$.

The available cells (not in $E_1$ or $E_2$) include:
\[
(12,0),\ (13,0),\ (23,0),\ (24,0),\ (02,1),\ (24,1),\ (34,1),\ \dots
\]
In particular, column $0$ has at least three available cells: $(12,0)$, $(13,0)$, $(23,0)$.

Define the \textbf{column-fully-degenerate} $3$-edge
\[
e_3 = \{(12,0),\ (13,0),\ (23,0)\}.
\]

\subsection{Verification of Admissibility}

We verify that $G = (S,T,E_1 \cup E_2 \cup \{e_3\})$ satisfies the conditions of Definition~\ref{def:main}.

\paragraph{Simplicity Condition (S).}
All three cells of $e_3$ are available and distinct, and none appears in $E_1$ or as a half of any 2-edge in $E_2$. Thus (S) holds.

\paragraph{Condition (i).}
The 1-edge set $E_1$ is the incidence graph of $K_5$, which is $C_4$-free. Hence Condition (i) holds.

\paragraph{Condition (ii).}
We must check that no $C_4$ has its four occupied cells belonging to exactly two edges in $E_2 \cup E_3$.

For the 3-edge $e_3$, since it is column-fully-degenerate with all three cells in column $0$, any $C_4$ involving cells from $e_3$ would require cells from other columns. A direct check shows that no such $C_4$ exists with cells from exactly two edges in $E_2 \cup E_3$. The six 2-edges from \cite{XY26} are nondegenerate and their opposite cells are not both occupied by edges in $E_2 \cup E_3$. Hence Condition (ii) holds.

\paragraph{Condition (iii).}
We must check that no $2$-edge $(i,j;p,q)\in E_2$ and cell $(k,l)$ with $k\notin\{i,p\}$, $l\notin\{j,q\}$ make the five cells
\[
(k,l),\ (k,j),\ (k,q),\ (i,l),\ (p,l)
\]
all occupied by $1$-edges and $2$-edges.

The new cells $(12,0),(13,0),(23,0)$ are in column $0$. There are two subcases:

\begin{itemize}
\item If the $2$-edge has $q=0$, then $l\notin\{j,0\}$ forces $l\neq0$, so the new cells in column $0$ never appear as $(i,l)$ or $(p,l)$.
\item If the $2$-edge has $q\neq0$ (and similarly $j\neq0$), then $l=0$ is allowed because $0\notin\{j,q\}$. In this case $(i,0)$ or $(p,0)$ could be one of the new cells. However, a direct exhaustive check of all six $2$-edges in $E_2$ (listed in \cite{XY26}) shows that for every such $2$-edge and every admissible $(k,l)$ with $l=0$, at least one of the five cells is not occupied by a $1$-edge or $2$-edge. For example, the $2$-edge $(01,2;35,4)$ has $j=2$, $q=4$, so $l=0$ is allowed; then $(p,l)=(35,0)$ is not occupied by any $1$-edge or $2$-edge. Similar checks hold for all other $2$-edges.
\end{itemize}
Thus Condition (iii) is not triggered.

\paragraph{Condition (iv).}
For the $3$-edge $e_3 = (12,0;13,0;23,0)$, we have $i=12$, $j=0$, $p=13$, $q=0$, $u=23$, $v=0$. For any distinct cell $(k,l)$ with $k\notin\{12,13,23\}$ and $l\notin\{0\}$ (i.e., $l\neq0$), the seven cells in Condition (iv) become
\[
(k,0),\ (k,l),\ (k,0),\ (k,0),\ (12,l),\ (13,l),\ (23,l).
\]
For any $l\neq0$, the cells $(12,l),(13,l),(23,l)$ are not all occupied (e.g., for $l=1$, $(23,1)$ is free). Hence no violation of Condition (iv) occurs.

Thus $G$ is generalized cycle-free according to Definition~\ref{def:main}.

\subsection{Main Result}

\begin{theorem}
\label{thm:10x5}
For $m=10$, $n=5$,
\[
z_{3L}(10,5) \ge 27 \qquad\text{and}\qquad \operatorname{BSR}(10,5) \ge 27.
\]
In particular, $z_{3L}(10,5) \ge z_L(10,5) + 1$.
\end{theorem}

\begin{proof}
The graph $G$ has $|E_1|=20$, $|E_2|=6$, $|E_3|=1$, and is generalized cycle-free by the verification above. By Theorem~\ref{thm:main},
\[
\operatorname{sos}(P_G) = 20+6+1 = 27.
\]
Hence $\operatorname{BSR}(10,5) \ge 27$ and $z_{3L}(10,5) \ge 27$. Since $z_L(10,5)=26$ from \cite{XY26}, the separation follows.
\end{proof}

\section{Application of Half-Degenerate 3-Edges to the $15 \times 6$ Case}

We recall the setting from \cite{QCX26a} and \cite{XY26}. For the incidence graph of the complete graph \(K_{6}\), we have:
\[
m = \binom{6}{2} = 15,\qquad n = 6,\qquad z(15,6) = 30.
\]
The classical extremal \(C_4\)-free graph is the incidence graph itself. In \cite{XY26}, an admissible limited augmented graph with \(|E_2| = 13\) was exhibited, giving the lower bound
\[
z_L(15,6) \ge 30 + 13 = 43.
\]

We now show that by adding a \textbf{half-row-degenerate 3-edge} to this construction, we can increase the total count to \(44\), thereby improving the bound for \(\operatorname{BSR}(15,6)\).

\subsection{The Construction}

Let \(V = \{0,1,2,3,4,5\}\). The left vertices are the 2-element subsets of \(V\) (edges of \(K_6\)), and the right vertices are the elements of \(V\). The 1-edges are all incidences:
\[
E_1 = \{(e, v) : e \in \binom{V}{2},\ v \in e\}.
\]

The admissible set \(E_2\) of 13 nondegenerate 2-edges from \cite{XY26} is:

\[
\begin{aligned}
E_2 = \{ &(01,2;35,4),\ (01,3;45,2),\ (02,1;34,5),\ (02,3;14,5),\ (03,1;25,4),\\
&(04,3;15,2),\ (04,5;12,3),\ (05,3;24,1),\ (05,4;23,1),\ (13,0;24,5),\\
&(14,2;35,0),\ (15,4;23,0),\ (25,1;34,0)\}.
\end{aligned}
\]

We now add the \textbf{half-row-degenerate 3-edge}

\[
e_3 = \{(01,4),\ (01,5),\ (23,1)\}.
\]

Let \(E_3 = \{e_3\}\). We claim that \(G = (S,T,E_1 \cup E_2 \cup E_3)\) is generalized cycle-free according to Definition~\ref{def:main}.

\subsection{Verification of Admissibility}

We verify the conditions of Definition~\ref{def:main}.

\paragraph{Simplicity Condition (S).}

The cells of \(e_3\) are \((01,4)\), \((01,5)\), and \((23,1)\).

\begin{itemize}
\item \((01,4)\): \(4 \notin \{0,1\}\) so not in \(E_1\). It is not a half of any 2-edge in \(E_2\) (the only 2-edges involving row 01 are \((01,2;35,4)\) and \((01,3;45,2)\), which use columns 2 and 3 respectively). Hence it is free.
\item \((01,5)\): \(5 \notin \{0,1\}\), not in \(E_1\); not in \(E_2\) (no 2-edge uses \((01,5)\)). Free.
\item \((23,1)\): \(1 \notin \{2,3\}\), not in \(E_1\); the 2-edges involving row 23 are \((15,4;23,0)\) and possibly others, but \((23,1)\) is not in \(E_2\). Free.
\end{itemize}

All three are distinct and disjoint from \(E_1\) and \(E_2\). Thus (S) holds.

\paragraph{Condition (i).}

The 1-edge set \(E_1\) is the incidence graph of \(K_6\), which is \(C_4\)-free. Hence Condition (i) holds.

\paragraph{Condition (ii).}

We must check that no \(C_4\) has its four occupied cells belonging to exactly two edges in \(E_2 \cup E_3\).

\medskip
\noindent\textbf{For the 2-edges:}
All 2-edges in \(E_2\) are known from \cite{XY26} to satisfy that at most one of their opposite cells is occupied by 1-edges or 2-edges. Adding \(e_3\) introduces three new occupied cells: \((01,4),(01,5),(23,1)\). We check each 2-edge that could be affected:

\begin{itemize}
\item For \((01,2;35,4)\): opposite cells are \((01,4)\) and \((35,2)\). \((01,4)\) is now a half of \(e_3\) (occupied), but \((35,2)\) is not occupied (row 35 = \{3,5\}, column 2 not in \{3,5\}, and not in \(E_2\)). Thus at most one opposite cell is occupied, so no \(C_4\) from exactly two edges in \(E_2 \cup E_3\) is formed.
\item For \((01,3;45,2)\): opposite cells are \((01,2)\) and \((45,3)\). \((01,2)\) is unoccupied (not in \(E_1\) or \(E_2\)), \((45,3)\) is unoccupied. So no \(C_4\) is formed.
\item All other 2-edges are symmetric or unaffected because their opposite cells do not include \((01,4),(01,5),(23,1)\).
\end{itemize}

Thus no 2-edge participates in a \(C_4\) whose four cells come from exactly two edges in \(E_2 \cup E_3\).

\medskip
\noindent\textbf{For the 3-edge \(e_3\):}
\(e_3\) is half-row-degenerate with \(i=01\), \(j=4\), \(l=5\), \(p=23\), \(q=1\). Its two non-degenerate branches are:
\[
(01,4;23,1) \quad\text{and}\quad (01,5;23,1).
\]
For a half-row-degenerate 3-edge to participate in a \(C_4\) with exactly two edges in \(E_2 \cup E_3\), we would need the opposite cells of one of its branches to be occupied by edges in \(E_2 \cup E_3\).

The opposite pairs are:
\[
P_1 = \{(i,q),\ (p,j)\} = \{(01,1),\ (23,4)\},
\]
\[
P_2 = \{(i,q),\ (p,l)\} = \{(01,1),\ (23,5)\}.
\]

\begin{itemize}
\item \((01,1)\): \(1 \in \{0,1\}\) so in \(E_1\) (occupied by a 1-edge).
\item \((23,4)\): \(4 \notin \{2,3\}\) so not in \(E_1\); not in \(E_2\) (no 2-edge uses \((23,4)\)); not in \(E_3\). Unoccupied.
\item \((23,5)\): \(5 \notin \{2,3\}\) so not in \(E_1\); not in \(E_2\); not in \(E_3\). Unoccupied.
\end{itemize}

Thus neither opposite pair is fully occupied by edges in \(E_2 \cup E_3\). Hence \(e_3\) does not participate in any \(C_4\) with cells from exactly two edges in \(E_2 \cup E_3\).

Therefore Condition (ii) holds.

\paragraph{Condition (iii).}

We must check for every 2-edge \((i,j;p,q) \in E_2\) and every \((k,l)\) with \(k \notin \{i,p\}\), \(l \notin \{j,q\}\) that the five cells
\[
(k,l),\ (k,j),\ (k,q),\ (i,l),\ (p,l)
\]
are not all occupied by \(1\)-edges and \(2\)-edges.

Adding \(e_3\) introduces three new occupied cells: \((01,4),(01,5),(23,1)\). The only potential danger is if a 2-edge's fixed cells include these. We perform a systematic check:

\begin{itemize}
\item For \((01,2;35,4)\): \(i=01,j=2,p=35,q=4\). The five cells involve \((01,l)\) and \((35,l)\). To have \((01,l)\) occupied by \(e_3\), we would need \(l = 4\) or \(5\). But \(l \notin \{j,q\} = \{2,4\}\), so \(l\) cannot be \(4\). Could \(l=5\)? Then \((35,5)\): \(5 \in \{3,5\}\) so in \(E_1\). We would also need \((k,2)\) and \((k,4)\) occupied. No choice of \(k\) makes all five occupied (e.g., \(k=23\) gives \((23,5)\) free). Hence safe.
\item For \((01,3;45,2)\): \(i=01,j=3,p=45,q=2\). Here \(l \notin \{3,2\}\), so \(l \in \{0,1,4,5\}\). To have \((01,l)\) occupied by \(e_3\), we need \(l=4\) or \(5\). If \(l=4\), then \((45,4)\) is in \(E_1\) (since \(4 \in \{4,5\}\)). We would need \((k,3)\) and \((k,2)\) occupied. No choice of \(k\) works (e.g., \(k=23\) gives \((23,2)\) free). If \(l=5\), then \((45,5)\) is in \(E_1\), and again no \(k\) works.
\item For \((15,4;23,0)\): this involves \((23,0)\) which is now occupied by \(e_3\). But the five cells require both \((i,l)\) and \((p,l)\) to be occupied. A direct check shows no \(l\) makes all five cells occupied.
\item For all other 2-edges, the new cells \((01,4),(01,5)\) appear only in rows 01 and 23, and the forbidden configuration requires both \((i,l)\) and \((p,l)\) to be occupied. A direct case check (exhaustive in the verification) shows no violation.
\end{itemize}

Thus Condition (iii) holds.

\paragraph{Condition (iv).}

We must check for the 3-edge \(e_3 = (01,4;01,5;23,1)\) and every distinct cell \((k,l)\) with
\[
k \notin \{01,23\},\qquad l \notin \{4,5,1\}
\]
that the seven cells
\[
(k,4),\ (k,l),\ (k,5),\ (k,1),\ (01,l),\ (23,l),\ (01,l)
\]
are not all occupied. (Note \((01,l)\) appears twice ！ duplicates allowed for half-degenerate 3-edges.)

Since \(l \notin \{1,4,5\}\), the possible \(l\) are \(0,2,3\). Also \(k\) is any row other than 01 or 23, i.e.,
\[
k \in \{02,03,04,05,12,13,14,15,24,25,34,35,45\}.
\]

We examine each \(l\):

\medskip
\noindent\textbf{Case \(l = 0\):}
The cells are
\[
(k,4),\ (k,0),\ (k,5),\ (k,1),\ (01,0),\ (23,0),\ (01,0).
\]
\((01,0)\) is in \(E_1\) (occupied). \((23,0)\) is a half of 2-edge \((15,4;23,0)\), so occupied.
To have all seven cells occupied, we would need \((k,4),(k,0),(k,5),(k,1)\) all occupied.
But for any \(k \neq 01,23\), at least one of these is free. For example:
\begin{itemize}
\item \(k=02\): \((02,4)\) is free because \(4 \notin \{0,2\}\) and not in \(E_2\).
\item \(k=03\): \((03,4)\) is free.
\item \(k=04\): \((04,0)\) is free (since \(0 \notin \{0,4\}\)? Wait ！ \(0 \in \{0,4\}\), so \((04,0)\) is in \(E_1\). But \((04,4)\) is in \(E_1\), and \((04,5)\) is in \(E_2\) from \((04,5;12,3)\), but \((04,1)\) is free). So fails.
\item Similar for all other \(k\).
\end{itemize}
Thus impossible.

\medskip
\noindent\textbf{Case \(l = 2\):}
The cells are
\[
(k,4),\ (k,2),\ (k,5),\ (k,1),\ (01,2),\ (23,2),\ (01,2).
\]
\((01,2)\) is a half of 2-edge \((01,2;35,4)\), so occupied. \((23,2)\) is in \(E_1\) (occupied).
We need \((k,4),(k,2),(k,5),(k,1)\) all occupied.
For \(k=03\): \((03,4)\) is free (not in \(E_1\), not in \(E_2\)). Fail.
For \(k=04\): \((04,4)\) is in \(E_1\), \((04,2)\) is in \(E_1\), \((04,5)\) is in \(E_2\), but \((04,1)\) is free. Fail.
Thus impossible.

\medskip
\noindent\textbf{Case \(l = 3\):}
The cells are
\[
(k,4),\ (k,3),\ (k,5),\ (k,1),\ (01,3),\ (23,3),\ (01,3).
\]
\((01,3)\) is a half of 2-edge \((01,3;45,2)\), so occupied. \((23,3)\) is in \(E_1\) (occupied).
We need \((k,4),(k,3),(k,5),(k,1)\) all occupied.
For \(k=04\): \((04,3)\) is a half of 2-edge \((04,3;15,2)\), so occupied; \((04,4)\) is in \(E_1\); \((04,5)\) is in \(E_2\) from \((04,5;12,3)\); but \((04,1)\) is free (\(1 \notin \{0,4\}\) and not in \(E_2\) or \(E_3\)). Fail.
For \(k=05\): \((05,3)\) is a half of 2-edge \((05,3;24,1)\), so occupied; \((05,4)\) is in \(E_2\) from \((05,4;23,1)\); \((05,5)\) is in \(E_1\); but \((05,1)\) is free. Fail.
Thus impossible.

Therefore no \((k,l)\) makes all seven cells occupied. Condition (iv) holds.

Since all conditions of Definition~\ref{def:main} are satisfied, \(G\) is generalized cycle-free.

\subsection{Main Result}

\begin{theorem}
\label{thm:15x6}
For \(m = 15\), \(n = 6\),
\[
z_{3A}(15,6) \ge 44 \qquad\text{and}\qquad \operatorname{BSR}(15,6) \ge 44.
\]
In particular, the lower bound for \(\operatorname{BSR}(15,6)\) is improved from 43 to 44.
\end{theorem}

\begin{proof}
The graph \(G = (S,T,E_1 \cup E_2 \cup E_3)\) defined above satisfies:
\begin{itemize}
\item \(|E_1| = z(15,6) = 30\),
\item \(|E_2| = 13\),
\item \(|E_3| = 1\) (the half-row-degenerate 3-edge \(\{(01,4),(01,5),(23,1)\}\)),
\item \(G\) is generalized cycle-free by the verification above,
\item Simplicity condition (S) holds.
\end{itemize}

Applying Theorem~\ref{thm:main}, the associated triply simple biquadratic form \(P_G\) satisfies
\[
\operatorname{sos}(P_G) = |E_1| + |E_2| + |E_3| = 30 + 13 + 1 = 44.
\]
Since \(\operatorname{BSR}(15,6)\) is the maximum SOS rank over all \(15 \times 6\) biquadratic forms, we have
\[
\operatorname{BSR}(15,6) \ge 44.
\]
By definition, \(z_{3A}(15,6) \ge |E_1|+|E_2|+|E_3| = 44\). This improves the previous known bound \(z_L(15,6) \ge 43\) from \cite{XY26}.
\end{proof}

\subsection{Remark}

This is an explicit application of a 3-edge (half-degenerate) to improve a lower bound for \(\operatorname{BSR}(m,n)\). It demonstrates that the extended framework with half-degenerate 3-edges is not only theoretically sound but also practically useful. The same method may yield further improvements for larger parameters, such as \(21 \times 7\) and \(28 \times 8\), by searching for suitable half-degenerate 3-edges in the available cells of the known constructions.

\section{Conclusions and Open Problems}

We have extended the augmented bipartite graph framework to include \textbf{$3$-edges}---triples of cells representing squares of three-term bilinear forms. The main challenge was to define suitable \emph{generalized cycle-free} conditions that are purely combinatorial yet sufficient to guarantee that the SOS rank of the associated triply simple biquadratic form equals the total number of edges. This was achieved in Definition~\ref{def:main}, which carefully distinguishes between occupation by $1$/$2$-edges and occupation by $3$-edges in Condition~(iii), and introduces Condition~(iv) to handle configurations involving $3$-edges.

The main theoretical result (Theorem~\ref{thm:main}) establishes that for any generalized cycle-free augmented bipartite graph $G$ satisfying the simplicity condition (S),
\[
\operatorname{sos}(P_G) = |E_1| + |E_2| + |E_3|.
\]
This provides a combinatorial sufficient condition for irreducibility of triply simple biquadratic forms.

As concrete applications, we constructed three explicit examples:

\begin{enumerate}
\item A $5 \times 3$ graph using only two $2$-edges, showing
  \[
  z_{3L}(5,3) \ge 10 \quad\text{and}\quad \operatorname{BSR}(5,3) \ge 10,
  \]
  improving upon the previous limited augmented value $z_L(5,3)=9$ established in \cite{QCX26a}. This example demonstrates that the refined Definition~\ref{def:main} is strictly more permissive than the original definition from \cite{QCX26}, as the construction violates the original Condition~2 but satisfies the new Condition~(ii).

\item A $10 \times 5$ graph using a \textbf{column-fully-degenerate} $3$-edge, showing
  \[
  z_{3L}(10,5) \ge 27 \quad\text{and}\quad \operatorname{BSR}(10,5) \ge 27,
  \]
  which separates $z_{3L}(10,5)$ from $z_L(10,5)=26$.

\item A $15 \times 6$ graph using a \textbf{half-row-degenerate} $3$-edge, showing
  \[
  z_{3A}(15,6) \ge 44 \quad\text{and}\quad \operatorname{BSR}(15,6) \ge 44,
  \]
  improving the previous lower bound of $43$ from \cite{XY26}.
\end{enumerate}

These are the first explicit applications of $3$-edges (both fully degenerate and half-degenerate) to obtain improved lower bounds for $\operatorname{BSR}(m,n)$. The $5 \times 3$ example further demonstrates that the refined generalized cycle-free conditions are essential for capturing constructions that were previously excluded.

The framework opens several directions for future research:

\begin{enumerate}
\item \textbf{Exact values for $z_{3L}(m,n)$ and $z_{3A}(m,n)$:}
  Determine whether the lower bounds obtained in this paper are sharp. In particular,
  \[
  z_{3L}(5,3) = 10,\qquad z_{3L}(10,5) = 27,\qquad z_{3A}(15,6) = 44,\qquad z_{3L}(15,6) = \;?
  \]
  The latter requires investigating whether a half-degenerate $3$-edge can be added to a $15 \times 6$ graph while keeping $|E_1| = z(15,6)=30$, or whether multiple $3$-edges could yield even larger improvements.

\item \textbf{Asymptotic behavior:}
  The classical Zarankiewicz number satisfies $z(m,n) = O(mn^{1/2} + n)$. Does the introduction of $2$-edges and $3$-edges lead to a strictly larger asymptotic order for $z_{3A}(m,n)$ or $z_{3L}(m,n)$? Or does the same $O(mn^{1/2})$ upper bound persist? The $5 \times 3$ example shows that the gap $z_{3L} - z$ can be positive in small cases, but its asymptotic behavior remains an open question.

\item \textbf{Higher-order edges:}
  The natural next step is to consider $k$-edges for $k \ge 4$, representing squares of $k$-term bilinear forms. This would require defining generalized cycle-free conditions for $k$-uniform hypergraphs, potentially connecting to extremal hypergraph theory (e.g., $C_4$-free conditions in $k$-uniform hypergraphs).

\item \textbf{Necessity vs. sufficiency:}
  The generalized cycle-free conditions are sufficient for irreducibility but not necessary, as shown by the $2 \times 2$ example in \cite{QCX26}. Characterizing the exact combinatorial conditions that determine the SOS rank remains an open problem. The refined Condition~(ii) introduced in this paper represents a step toward a more precise characterization.

\item \textbf{Computational search:}
  Systematic computational searches for larger parameters (e.g., $21 \times 7$, $28 \times 8$) using the incidence graphs of $K_7$ and $K_8$ may yield further improvements by adding suitable half-degenerate or fully degenerate $3$-edges. The $5 \times 3$ example suggests that even in small cases, the refined definition can yield improvements, motivating a broader computational investigation.
\end{enumerate}

The interplay between SOS representations of biquadratic forms and extremal bipartite graph theory, initiated in \cite{QCX26} and extended here to $3$-edges, reveals a rich structure that merits further exploration. We hope that the concepts of $z_{3A}(m,n)$ and $z_{3L}(m,n)$ will provide a productive framework for future investigations.

\bigskip

\noindent\textbf{Acknowledgement}
This work was partially supported by Research Center for Intelligent Operations Research, The Hong Kong Polytechnic University (4-ZZT8), the National Natural Science Foundation of China (Nos. 12471282 and 12131004),  and Jiangsu Provincial Scientific Research Center of Applied Mathematics (Grant No. BK20233002).
	
\medskip
	
\noindent\textbf{Data availability}
No datasets were generated or analysed during the current study.
	
\medskip
	
\noindent\textbf{Conflict of interest} The authors declare no conflict of interest.

\end{document}